\documentclass[10pt]{article}

\usepackage{authblk}
\usepackage[margin=0.9in]{geometry}

\usepackage{graphicx}
\usepackage{multicol,multirow}
\usepackage{amsmath,amssymb,amsfonts}
\usepackage{mathrsfs}
\usepackage[numbers]{natbib}
\usepackage{amsthm}
\newtheorem{theorem}{Theorem}[section]
\newtheorem{lemma}[theorem]{Lemma}

\newtheorem{defn}{Definition}[section]
\newtheorem{prop}{Proposition}[section]
\newtheorem{rem}[theorem]{Remark}
\newtheorem{ex}[theorem]{Example}

\numberwithin{equation}{section}

\title{Existence Results for double phase problem in Sobolev-Orlicz spaces with variable exponents in Complete Manifold}

\author{Ahmed Aberqi$^{1}$, Jaouad Bennouna$^{2}$, Omar Benslimane$^{2}$,\\ Maria Alessandra Ragusa$^{3}$\\

$^{1}${\small Laboratory LAMA, Sidi Mohamed Ben Abdellah University, National School of Applied Sciences, Fez, Morocco,}\\

$^{2}${\small Laboratory LAMA, Department of Mathematics, Sidi Mohamed Ben Abdellah University, Faculty of Sciences Dhar El Mahraz, B.P 1796 Atlas Fez, Morocco,}\\

$^{3}${\small Dipartimento di Matematica e Informatica, Università di Catania, Catania, Italy\protect\\ RUDN University, 6 Miklukho-Maklay St, 117198, Moscow, Russia}
}

\begin{document}
\maketitle



%
%




\begin{abstract}
In this paper, we study the existence of non-negative non-trivial solutions for a class of double-phase problems where the source term is a Caratheodory function that satisfies the Ambrosetti-Rabinowitz type condition in the framework of Sobolev-Orlicz spaces with variable exponents in complete compact Riemannian n-manifolds. Our approach is based on the Nehari manifold and some variational techniques. Further-more, the Hölder inequality, continuous and compact embedding results are proved.
\end{abstract}

\bigskip
\bigskip
\noindent {\textit{Mathematics Subject Classification(2000):}\thinspace
\thinspace 35J20, 35J47, 35J60.} \newline


\section[This is an A Head]{Introduction}
Let $ ( M,\,g )$ be a smooth, complete compact Riemannian n-manifold. In this paper, we focused on the existence of non-trivial solutions of the following double phase problem
$$ ( \mathcal{P} ) \begin{cases}
\, -\mbox{div}  (\, | \, \nabla u ( x )\, |^{p( x ) -2}  \nabla u + \mu ( x ) \, |\, \nabla u ( x ) \,|^{q( x ) - 2} \, \nabla u )\\ \hspace*{0.3cm} =  \lambda |\, u( x ) \, |^{q( x ) - 2} \, u( x ) - |\, u( x )\,|^{p( x ) - 2} \, u( x ) + f ( x,  u( x ) ) & \text{ in $ M$ }, \\[0.3cm]
\, u \,= \, 0 & \text{ on $ \partial M $ },
\end{cases} $$
where $-  \Delta_{p( x )} u( x ) = - \,div\,( \,| \, \nabla u( x ) \, |^{p( x ) - 2} \,.\, \nabla u( x ) ),$ $-  \Delta_{q( x )} u( x ) = \\- \,div\,( \,| \, \nabla u( x ) \, |^{q( x ) - 2} \,.\, \nabla u( x ) ) $ are the $ p( x ) $-laplacian and $ q( x ) $-laplacian in $( M, g )$ respectively, $ \lambda > 0$ is a parameter specified later, the function $ \mu : \overline{M} \rightarrow \mathbb{R}^{+}_{*}$ is supposed to be Lipschitz continuous, and the variables exponents $ p, \, q \in C( \overline{M} )$  satisfy the assumption \eqref{1.1} in section \ref{section3}. \\\
The perturbation $f ( x, u )$ is a Caratheodory function which satisfies the Ambrosetti-Rabinowitz type condition:\\
$ ( f_{1} ): $ There exists $ \beta > p^{+} $ and some $A > 0$ such as for each $ |\, \alpha\,| > A$ we have
$$ 0 < \int_{M} F( x, \alpha ) \,\,dv_{g} ( x ) \leq \int_{M} f( x, \alpha ) \,.\, \frac{\alpha}{\beta} \,\, dv_{g} ( x ) \,\, \mbox{a.e} \,\,x \in M, $$
where $ \displaystyle F( x, \alpha ) = \int_{0}^{\alpha} f( x, t ) \,\, dt $ being the primitive of $ f( x, \alpha )$ and $dv_{g} = \sqrt{\mbox{det} ( g_{ij} )} \,\, dx $ is the Riemannian volume element on $( M, g ),$ where the $ g_{ij}$ are the components of the Riemannian metric $g$ in the chart and $dx$ is the Lebesgue volume element of $\mathbb{R}^{N}.$\\
$ ( f_{2} ): \, f( x, 0 ) = 0. $\\
And\\
$ ( f_{3} ): \, \displaystyle \lim_{| \, \alpha \, | \rightarrow 0} \frac{f( x, \alpha )}{| \, \alpha \, |^{q( x ) - 1}} = 0 $ uniformly a.e $ x \in M $.\\

Up to this day, several contributions have been devoted to study double phase problems. This kind of operator was introduced, first, by Zhikov in his relevant paper \cite{zhikov1987averaging} in order to describe models with strongly anisotropic materials by studying the functional $$ u \longmapsto \int_{\Omega} ( |\, \nabla u\,|^{p} + \mu ( x ) \, |\, \nabla u\,|^{q} ) \,\, dx, $$ where $ 1 < p < q < N$ and with a nonnegative weight function $\mu \in L^{\infty} ( \Omega ),$ see also the works of Zhikov \cite{zhikov1995lavrentiev, zhikov1997some} and the monograph of Zhikov-Kozlov-Oleinik \cite{jikov2012homogenization}. Indeed, we can easily see that the previous function reduces to p-laplacian if $\mu ( x ) = 0$ or to the weighted laplacian $( p( x ), \, q( x ) )$ if $ \displaystyle\inf_{x \in \overline{M}} \mu ( x ) > 0,$ respectively.\\
In the case of single valued equations, Liu-Dai in \cite{liu2018existence} discussed double phase problems and proved the existence and the multiplicity of the results, with the sign-changing solutions by variational method. A similar treatment has been recently done by Gasi{\'n}ski and Papageorgiou in \cite{gasinski2019constant} via the Nehari manifold method. Following this direction, Papageorgiou, N. S. and Repov{\v{s}}, D. D. and Vetro, C. in \cite{papageorgiou2020positive} studied the existence of positive solutions for a class of double phase Dirichlet equations which has the combined effects of a singular term and of a parametric super-linear term. In particular, in \cite{tachikawa2020boundary} the author provides the Hölder continuity up to the boundary of minimizers of so-called double phase functional with variable exponents, under suitable Dirichlet boundary conditions. For more details, we refer the reader to \cite{aubin1982nonlinear, cencelj2018double, gaczkowski2016sobolev, hebey2000nonlinear, guo2015dirichlet, ragusa2019regularity, marino2020existence, shi2020multiple} and the references therein.

Also, there are many articles on nonstandard growth problems, especially on $p( x )$-growth and double phase problems. About $p( x )$-growth problems, see  \cite{aberqi2019existence, aberqi2017nonlinear, benslimane2020existence, A, benslimane2021some, benslimane2020existence1, boccardo1993nonlinear, boccardo1992nonlinear, fan2001compact, fan2001spaces, ragusa2005partial} and the references given there.\\

Studying this type of problems is both significant and relevant. In the one hand, we have the physical motivation; since the double phase operator has been used to model the steady-state solutions of reaction-diffusion problems, that arise in biophysic, plasma-physic and in the study of chemical reactions. In the other hand, these operators provide a useful paradigm for describing the behaviour of strongly anisotropic materials, whose hardening properties are linked to the exponent governing the growth of the gradient change radically with the point, where the coefficient $\mu ( . )$ determines the geometry of a composite made of two different materials. \\

Motivated by the aforementioned works, the aim of this paper, is to prove the existence of non-negative non-trivial solutions of the problem $( \mathcal{P} )$ where the perturbation $f( x, u )$ is a Caratheodory function, that satisfies the Ambrosetti-Rabinowitz type condition. To the best of our knowledge, the existence result for double-phase problems $( \mathcal{P} )$ in the framework of Sobolev-Orlicz spaces with variable exponents in complete manifold has not been considered in the literature. The present paper is the first study devoted to this type of problem in the setting of Sobolev-Orlicz spaces with variable exponents in a complete manifold.\\

We would like to draw attention to the fact that the $p ( x )-$laplacian operator has more complicated non-linearity than the p-laplacian operator. For example, they are non-homogeneous. Thus, we cannot use the Lagrange Multiplier Theorem in many problems involving this operators, which prove that our problem is more difficult than the operators p-Laplacian type.\\

The paper is organized as follows. In section \ref{section 2}, we recall the most important and relevant properties and notations of Lebesgue spaces with variable exponents and Sobolev-Orlicz spaces with variable exponents in complete manifold. Moreover, we show two new results: the first one, the Hölder inequality and the second one, the embedding result of these spaces into Lebesgue space with variable exponent. In section \ref{section3}, we introduce the Nehari manifold associated with $( \mathcal{P} )$ and we study three parts, corresponding to local minima, local maxima and the points of inflection. Finally, in section \ref{section4}, we demonstrate the existence of two non-negative non-trivial solutions of the problem $( \mathcal{P} ).$

\section{Mathematical background and auxiliary results}\label{section 2}
In this section, we recall the most important and relevant properties and notations about Sobolev spaces with variable exponents and Sobolev spaces with variable exponents on manifolds, and we prove some properties, that we will need in our analysis of the problem $ ( \mathcal{P} )$, by that, referring to \cite{aubin1982nonlinear, hebey2000nonlinear, fan2001spaces, radulescu2015partial, benslimane2020existence1, gaczkowski2016sobolev} for more details.
\subsection{Sobolev spaces with variable exponents}
Let $ \Omega$ be a bounded open subset of $ \mathbb{R}^{N} \, ( N \geq 2 )$, we define the Lebesgue space with variable exponent $ L^{q(\cdot)} ( \Omega )$ as the set of all measurable function $ u : \Omega \longmapsto \mathbb{R} $ for which the convex modular
$$ \rho _{q(\cdot)} ( u ) = \int_{\Omega} |\, u( x )\,|^{q( x )} \,\, dx, $$ is finite. If the exponent is bounded, i.e if $ q^{+} = ess \,sup \{ \, q( x ) / x \in \Omega \, \} < + \infty,$ then the expression $$ ||\, u \,||_{q(\cdot)} = \inf \{ \, \lambda > 0: \, \rho_{q(\cdot)} \bigg( \frac{u}{\lambda} \bigg) \leq 1 \, \}, $$ defines a norm in $ L^{q(\cdot)} ( \Omega )$, called the Luxemburg norm. The space $ ( L^{q(\cdot)} ( \Omega ), \, ||\,.\,||_{q(\cdot)} )$ is a separable Banach space. Moreover, if $ 1 < q^{-} \leq q^{+} < +\infty,$ then $ L^{q(\cdot)} ( \Omega )$ is uniformly convex, where $ q^{-} = ess\, inf \{ \, q( x ) / x \in \Omega \, \},$ hence reflexive, and its dual space is isomorphic to $ L^{q^{'}(\cdot)} ( \Omega )$ where $ \frac{1}{q( x )} + \frac{1}{q^{'} ( x )} = 1.$\\
Finally, we have the Hölder type inequality:
$$ \bigg|\, \int_{\Omega} u\, v \,\, dx \, \bigg| \leq \bigg( \, \frac{1}{q^{-}} + \frac{1}{( q^{'} )^{-}} \bigg) \, ||\, u\,||_{q(\cdot)} ||\, v\,||_{q^{'} (\cdot)},$$ for all $ u \in L^{q(\cdot)} ( \Omega )$ and $ v \in L^{q^{'}(\cdot)} ( \Omega )$.\\
Now, we define the variable exponent Sobolev space by $$ W^{1, q( \cdot )} ( \Omega ) = \{ u \in L^{q( \cdot )} ( \Omega ) \,\, \mbox{and} \,\, | \nabla u | \in L^{q( \cdot )} (\Omega ) \}.$$ which is a Banach space equipped with the following norm $$ || u ||_{1, q( \cdot )} = || u ||_{q( \cdot )} + || \nabla u ||_{q(\cdot)} \hspace*{1cm} \forall u \in W^{1, q(\cdot )} ( \Omega ).$$
The space $( W^{1, q( \cdot )} ( \Omega ), || \cdot ||_{1, q( \cdot )} )$ is separable and reflexive Banach space.\\
We denote by $W_{0}^{1, q(\cdot)} (\Omega)$ the closure $C_{0}^{\infty} ( \Omega )$ in $ W^{1, q( \cdot )} ( \Omega ).$
\begin{prop} \cite{fan2001compact} ( Poincaré inequality )
If $ q \in C_{+} ( \overline{\Omega} ),$ then there is a constant $c > 0$ such that
$$ ||\, u\,||_{q( x )} \leq c \, || \, |\, \nabla u\,|\, ||_{q( x )}, \hspace*{0.5cm} \forall u \in W_{0}^{1, \, q( x )} ( \Omega ).$$
Where, $ \displaystyle C_{+} ( \overline{\Omega} ) = \{ q/ q \in C( \overline{\Omega} ), \, q( x ) > 1 \,\, \mbox{for} \,\, x \in \overline{\Omega} \}.$ \\
Consequently, $ ||\, u\,|| = || \, |\, \nabla u \,|\, ||_{q( x )}$ and $||\, u\,||_{1,\, q( x )}$ are equivalent norms on $W_{0}^{1, q( x )} ( \Omega ).$
\end{prop}
\subsection{Sobolev spaces with variable exponents on manifolds}
In the following, all the manifolds we consider are smooth, and we will use the following conditions on $( M, g )$, depending on the context:
\begin{defn}
Let $ ( M, g ) $ be a smooth Riemannain n-manifolds and let $ \nabla $ be the Levi-Civita connection. If $u$ is a smooth function on $M$, then $ \nabla^{k} u $ denotes the $k-$th covariant derivative of $u$, and $ | \, \nabla^{k} u \, | $ the norm of $ \nabla^{k} u $ defined in local coordinates by
$$ | \, \nabla^{k} u \, |^{2} = g^{i_{1} j_{1}} \cdots g^{i_{k} j_{k}} \, ( \nabla^{k} u )_{i_{1} \cdots i_{k}} \, ( \nabla^{k} u )_{j_{1} \cdots j_{k}} $$
where Einstein's convention is used.
\end{defn}
\begin{rem}
A smooth manifold $M$ of dimension $n$ is a connected topological manifold $M$ of dimension $n$ together with a $C^{\infty} -$complete atlas.
\end{rem}
\begin{ex}
The following examples are classical examples of smooth manifolds:
$\bullet \hspace*{0.4cm}$ The Euclidean space $\mathbb{R}^{n}$ itself.\\
$\bullet \hspace*{0.4cm}$ The torus $T^{n}.$\\
$\bullet \hspace*{0.4cm}$ The unit sphere $S^{n}$ of $ \mathbb{R}^{n + 1}.$\\
$\bullet \hspace*{0.4cm}$ The real projective space $\mathbb{P}^{n} ( \mathbb{R} ).$
\end{ex}
\begin{defn}
To define variable Sobolev spaces, given a variable exponent $q$ in $ \mathcal{P} ( M ) $ ( the set of all measurable functions $p(\cdot) : M \rightarrow ( 1, \infty )$ ) and a natural number $k$, introduce
$$ C^{q(\cdot)}_{k} ( M ) = \{ \, u \in C^{\infty} ( M ) \,\, \mbox{such that } \,\, \forall j \,\, 0 \leq j \leq k \,\, | \, \nabla^{j} u \, | \in L^{q( \cdot ) } ( M )\, \}. $$
On $ C^{q( \cdot )}_{k} ( M ) $ define the norm
$$ || \, u \, ||_{L^{q( \cdot )}_{k} ( M )} = \sum_{j = 0}^{k} || \, \nabla^{j} u \, ||_{L^{q( \cdot )} ( M )}. $$
\end{defn}
\begin{defn}
The Sobolev spaces $ L_{k}^{q( \cdot )} ( M ) $ is the completion of $ C^{q(\cdot)}_{k} ( M ) $ with respect to the norm $ || \, u \, ||_{L^{q( \cdot )}_{k}}$. If $\Omega$ is a subset of $M,$ then $L_{k, 0}^{q( \cdot )} ( \Omega )$ is the completion of $C_{k}^{q( \cdot )} ( M ) \cap C_{0} ( \Omega )$ with respect to $ || . ||_{L^{q( \cdot )}_{k}},$ where $C_{0} ( \Omega )$ denotes the vector space of continuous functions whose support is a compact subset of $\Omega.$
\end{defn}
\begin{defn}
Given $ ( M, g ) $ a smooth Riemannian manifold, and $ \gamma : \, [\, a, \, b \, ] \longrightarrow M $ is a curve of class $ C^{1} $. The length of $ \gamma $ is
$$ l( \gamma ) = \int_{a}^{b} \sqrt{g \, ( \, \frac{d \gamma }{d t }, \, \frac{d \gamma}{d t}\, )} \,\,dt, $$
and for a pair of points $ x, \, y \in M$, we define the distance $ d_{g} ( x, y ) $ between $x$ and $y$ by
$$ d_{g} ( x, y ) = \inf \, \{ \, l( \gamma ) : \, \gamma: \, [ \, a, \, b \,] \rightarrow M \,\, \mbox{such that} \,\, \gamma ( a ) = x \,\, \mbox{and} \,\, \gamma ( b ) = y \, \}. $$
\end{defn}
\begin{defn}
A function $ s: \, M \longrightarrow \mathbb{R} $ is log-Hölder continuous if there exists a constant $c$ such that for every pair of points $ \{ x, \, y \} $ in $ M$ we have
$$ | \, s( x ) - s( y ) \, | \leq \frac{c}{log ( e + \frac{1}{d_{g} ( x, y )} \, )}. $$
We note by $ \mathcal{P}^{log} ( M ) $ the set of log-Hölder continuous variable exponents. The relation between $ \mathcal{P}^{log} ( M ) $ and $ \mathcal{P}^{log} ( \mathbb{R}^{N} ) $ is the following:
\end{defn}
\begin{prop} \cite{aubin1982nonlinear, gaczkowski2016sobolev}
Let $ q \in \mathcal{P}^{log} ( M ) $, and let $ ( \Omega, \phi ) $ be a chart such that
$$ \frac{1}{2} \delta_{i j } \leq g_{i j} \leq 2 \, \delta_{i j } $$
as bilinear forms, where $ \delta_{i j} $ is the delta Kronecker symbol. Then $ qo\phi^{-1} \in \mathcal{P}^{log} ( \phi ( \Omega ) ).$
\end{prop}
\begin{prop} ( Hölder's inequality )
For all $ u \in L^{q( \cdot )} ( M )$ and $ v \in  L^{q'( \cdot )} ( M )  $ we have $$ \int_{M} | \, u( x ) \, v( x ) \, | \,\, dv_{g} ( x ) \leq r_{q} \, || \, u \,||_{L^{q( \cdot )} ( M )} \,.\, || \, v \, ||_{L^{q'( \cdot )} ( M )}.$$
Where $r_{q}$ be a positive constant depend to $q^{-}$ and $q^{+}.$
\end{prop}
\begin{proof}
Obviously, we can suppose that $|| u ||_{L^{q(x)} ( M )} \neq 0 $ and $|| v ||_{L^{q'(x)} ( M )} \neq 0, $ we have
$$ 1 < q( x ) < \infty, \,\, | u( x ) | < \infty, \,\, | v( x ) | < \infty \,\,\mbox{a.e} \,\, x \in M. $$
By young inequality, we have
\begin{align*}
\frac{u( x ) \,.\, v( x )}{|| u( x ) ||_{L^{q( x )} ( M )} \,.\,|| v( x ) ||_{L^{q'( x )} ( M )}} \leq & \, \frac{1}{q( x )} \, \bigg( \frac{| u( x ) |}{|| u( x ) ||_{L^{q( x )} ( M )}} \bigg)^{q( x )} \\&+ \frac{1}{q'( x )} \, \bigg( \frac{| v( x ) |}{|| v( x ) ||_{L^{q'( x )} ( M )}} \bigg)^{q'( x )}
\end{align*}
Integrating over $M$, we obtain
\begin{align*}
&\int_{M} \frac{|\,u( x ) \,.\, v( x )\,|}{|| u( x ) ||_{L^{q( x )} ( M )} \,.\,|| v( x ) ||_{L^{q'( x )} ( M )}} \,\, dv_{g} ( x ) \\ & \leq \frac{1}{q^{-}} \, \int_{M} \bigg( \frac{| u( x ) |}{|| u( x ) ||_{L^{q( x )} ( M )}} \bigg)^{q( x )} \,\, dv_{g} ( x ) \\& \hspace*{0.3cm} + \big( 1 - \frac{1}{q^{+}} \big) \, \int_{M} \bigg( \frac{| v( x ) |}{|| v( x ) ||_{L^{q'( x )} ( M )}} \bigg)^{q'( x )} \,\, dv_{g} ( x ) \\& \leq 1 + \frac{1}{q^{-}} - \frac{1}{q^{+}},
\end{align*}
then, using the same technique as in the proof of Theorem 1.15 in \cite{fan2001spaces}, we get that 
\begin{align*}
\int_{M} |\, u( x ) \,.\, v( x ) \, |\,\, dv_{g} ( x ) &\leq \big( 1 + \frac{1}{q^{-}} + \frac{1}{q^{+}} \big) \, || u( x ) ||_{L^{q( x )} ( M )} \,.\, || v( x ) ||_{L^{q'( x )} ( M )}  \\& \leq r_{q} || u( x ) ||_{L^{q( x )} ( M )} \,.\, || v( x ) ||_{L^{q'( x )} ( M )},
\end{align*}
Which complete the proof.
\end{proof}
\begin{rem}
If $a$ and $b$ are two positive functions on $M$, then by Hölder's inequality and \cite{fan2001compact, gaczkowski2016sobolev} we have
\begin{equation}\label{2.1}
\int_{q^{-} < 2} a^{\frac{q^{-}}{2}} \, b^{\frac{2q^{-} - {q^{-}}^{2}}{2}} \, \leq 2  \, || \, \mathbb{1}_{q^{-} < 2} \,\, a^{\frac{q^{-}}{2}} \, ||_{L^{\frac{2}{q^{-}}}} \, .\, || \, \mathbb{1}_{q^{-} < 2} \,\, b^{\frac{2q^{-} - {q^{-}}^{2}}{2}} \, ||_{L^{\frac{2}{ 2 - q^{-}}}}.
\end{equation}
where $ \mathbb{1} $ is the indicator function of $ M $, moreover, since
$$ ||\, \mathbb{1}_{q^{-} < 2} \,\,a^{\frac{2}{q^{-}}} \, ||_{L^{\frac{2}{q^{-}}}} \leq \max \{ \, \rho_{1} ( a ), \, \rho_{1} ( a )^{\frac{q^{-}}{2}} \, \} $$
and
$$ || \, \mathbb{1}_{q^{-} < 2} \,\, b^{\frac{2q^{-} - {q^{-}}^{2}}{2}} \, ||_{L^{\frac{2}{ 2 - q^{-}}}} \leq \max \{ \, \rho_{q( \cdot )} ( b )^{\frac{2 - q^{-}}{2}}, \, 1 \, \},$$
we get,
\begin{equation}\label{2.2}
\int_{q^{-} < 2} a^{\frac{q^{-}}{2}} \, .\, b^{\frac{2q^{-} - {q^{-}}^{2}}{2}} \leq 2 \, \max \, \{ \, \rho_{1} ( a ), \rho_{1} ( a )^{\frac{q^{-}}{2}} \, \} \, \max \, \{ \, \rho_{q( \cdot )} ( b )^{\frac{2 - q^{-}}{2}}, \, 1 \, \}.
\end{equation}
\end{rem}
\begin{defn}
We say that the Riemannian n-manifold $ ( M, g ) $ has property $ B_{vol} ( \lambda, v ) $ where $\lambda$ is a constant, if its geometry is bounded in the following sense:\\
$ \hspace*{1cm} \bullet \,\, \mbox{The Ricci tensor of g noted by Rc ( g ) verifies,} \,\,Rc ( g ) \geq \lambda ( n - 1 ) \, g $ for some $ \lambda,$ where $n$ is the dimension of $M.$\\
$ \hspace*{1cm} \bullet  $ There exists some $ v > 0 $ such that $ | \, B_{1} ( x ) \, |_{g} \geq v \,\, \forall x \in M,$ where $B_{1} ( x ) $ are the balls of radius 1 centered at some point $x$ in terms of the volume of smaller concentric balls.
\end{defn}
\begin{rem}
If $ M = \Omega \subseteq  \mathbb{R}^{N} $ is a bounded open set, then the following inequality is related to the two exponents $p, \, q $ ( isotropic case )
$$ \frac{q}{p} < 1 + \frac{1}{N}.$$
This condition is essential, among others, for the embeddings of spaces to be satisfied.
\end{rem}
\begin{prop} \cite{aubin1982nonlinear,hebey2000nonlinear} \label{prop10}
Let $ ( M, g ) $ be a complete compact Riemannian n-manifold. Then, if the embedding $ L^{1}_{1} ( M ) \hookrightarrow L^{\frac{n}{n - 1}} ( M )$ holds, then whenever the real numbers $q$ and $p$ satisfy $$ 1 \leq q < n, $$ and $$ q \leq p \leq q* = \frac{n q}{n - q}, $$ the embedding $ L^{q}_{1} ( M ) \hookrightarrow L^{p} ( M ) $ also holds.
\end{prop}
\begin{prop} \cite{aubin1982nonlinear,hebey2000nonlinear}\label{prop6}
Assume that the complete compact Riemannian n-manifold $ ( M, g ) $ has property $ B_{vol} ( \lambda, v ) $ for some $ ( \lambda, v ).$ Then there exist positive constants $ \delta_{0} = \delta_{0} ( n, \, \lambda, \, v ) $ and $ A = A ( n, \, \lambda, \, v ) $, we have, if $ R \leq \delta_{0} $, if $ x \in M $ if $ 1 \leq q \leq n $, and if $ u \in L^{q}_{1,0} ( \, B_{R} ( x ) \, ) $ the estimate
$$ || \, u \, ||_{L^{p}} \leq A \,p \, || \, \nabla u \, ||_{L^{q}},$$ where $ \frac{1}{p} = \frac{1}{q} - \frac{1}{n}.$
\end{prop}
We can extend the above proposition from the case when exponents $p$ and $q$ are constant, when $p(\cdot)$ and $q(\cdot)$ are functions.
\begin{prop} \label{prop61}
Assume that for some $ ( \lambda, v ) $ the complete compact Riemannian n-manifold $( M, g ) $ has property $ B_{vol} ( \lambda, v ) $. Then there exist positive constants $ \delta_{0} = \delta_{0} ( n, \, \lambda, \, v ) $ and $ A = A ( n, \, \lambda, \, v ) $, we have, if $ R \leq \delta_{0} $, if $ x \in M $ if $ 1 \leq q(\cdot ) \leq n $, and if $ u \in L^{q( \cdot )}_{1,0} ( \, B_{R} ( x ) \, ) $ the estimate
$$ || \, u \, ||_{L^{p( \cdot )}} \leq A \,p^{-} \, || \, \nabla u \, ||_{L^{q( \cdot )}},$$ where $ \frac{p( \cdot )}{q( \cdot )} < 1 + \frac{1}{n}.$
\end{prop}
\begin{proof}
To demonstrate this Proposition, we use the same technique as proposition \ref{prop6}, for more detail see \cite{aubin1982nonlinear,hebey2000nonlinear}.
\end{proof}
In the following, we denote for all $ u \in W_{0}^{1, q( x )} ( M )$ that $$ \rho_{p( \cdot )} ( u ) = \int_{M} | u( x ) |^{p( x )} \,\, dv_{g} ( x ) \,\,\, \mbox{and} \,\,\, \rho_{q( \cdot )} ( u ) = \int_{M} | u( x ) |^{q( x )} \,\, dv_{g} ( x ).$$
\begin{prop} \cite{aubin1982nonlinear,hebey2000nonlinear, gaczkowski2016sobolev} \label{prop7}
Assume that for some $ ( \lambda, v ) $ the complete compact Riemannian n-manifold $( M, g ) $ has property $ B_{vol} ( \lambda, v ) $. Let $ p \in \mathcal{P} ( M ) $ be uniformly continuous with $ q^{+} < n.$ Then $ L^{q( \cdot )}_{1} ( M ) \hookrightarrow L^{p( \cdot )} ( M ), \, \forall q \in \mathcal{P} ( M ) $ such that $ q \ll p \ll q* = \frac{n q}{n - q}. $ In fact, for $ || \, u \, ||_{L^{q( \cdot )}_{1}} $ sufficiently small we have the estimate $$ \rho_{p( \cdot )} ( u ) \leq G \, ( \, \rho_{q( \cdot )} ( u ) + \rho_{q( \cdot )} ( | \, \nabla u \, | ) \, ), $$ where the positive constant $G$ depend on $ n, \, \lambda, \, v, \, q $ and $ p $.
\end{prop}
\begin{prop}\label{prop13} \cite{guo2015dirichlet}
Let $ u \in L^{q( x )} ( M ), \,\{ \, u_{k} \,\} \subset L^{q( x )} ( M ), \, k \in \mathbb{N},$ then we have
\begin{enumerate}
\item[(i)]  $|| u ||_{q( x )} < 1 \,\,\mbox{( resp. = 1, $>$ 1 )} \iff \rho_{q( x )} ( u ) < 1 \,\,\mbox{( resp. = 1, $>$ 1 )},$
\item[(ii)]  $ || u ||_{q( x )} < 1 \Rightarrow || u ||_{q( x )}^{q^{+}} \leq \rho_{q( x )} ( u ) \leq || u ||_{q( x )}^{q^{-}},$
\item[(iii)]  $ || u ||_{q( x )} > 1 \Rightarrow || u ||_{q( x )}^{q^{-}} \leq \rho_{q( x )} ( u ) \leq || u ||_{q( x )}^{q^{+}},$
\item[(iv)] $ \lim_{k \rightarrow + \infty} || u_{k} - u ||_{q( x )} = 0 \iff \lim_{k \rightarrow + \infty} \rho_{q( x )} ( u_{k} - u ) = 0. $
\end{enumerate}
\end{prop}
To compare the functionals $ ||\,\,\,\, ||_{q(\cdot)}$ and $ \rho_{q( \cdot )} (\, )$, one has the relation
$$ \min \{ \rho_{q(\cdot)}( u )^{\frac{1}{q^{-}}}, \, \rho_{q( \cdot )} ( u  )^{\frac{1}{q^{+}}} \} \leq || u ||_{L^{q(\cdot)}} \leq \max \{ \rho_{q( \cdot )} ( u )^{\frac{1}{q^{-}}}, \, \rho_{q( \cdot )} ( u )^{\frac{1}{q^{+}}} \}.$$
So, if $ q^{+} < n,$ we have the embedding
$$ L^{q( \cdot )}_{1, 0} ( B_{R} ( x ) ) \hookrightarrow L^{p( \cdot )} ( M ),$$ where, $ p( x ) = \frac{n q( x )}{n - q( x )}.$ In fact, there exists a positive constant $ D = D( n, \lambda, v, q^{+}, q^{-} )$ such as for every $u$ in $ L^{q( \cdot )}_{1, 0} ( B_{R} ( x ) ),$ we have by Poincaré inequality and Proposition \ref{prop13} that 
\begin{align*}
|| u ||_{L^{p( \cdot)} ( M )} & \leq D \, || u ||_{L^{q( \cdot )}_{1} ( M )} = D ( || u ||_{L^{q( \cdot )} ( M )} + || \nabla u ||_{L^{q( \cdot )} ( M )} ) \\& \leq D ( c + 1 ) \, || \nabla u ||_{L^{q( \cdot )} ( M )},
\end{align*}
where $c$ is the Poincaré constant. Hence, 
\begin{align} \label{Said}
\rho_{p( \cdot )} ( u ) & \leq || u ||^{p^{+}}_{L^{p( \cdot )} ( M )} \leq D^{p^{+}} ( c + 1 )^{p^{+}} \, || \nabla u ||^{p^{+}}_{L^{q( \cdot)} ( M )} \nonumber \\& \leq D^{p^{+}} ( c + 1 )^{p^{+}} \, \max \{ \rho_{q( \cdot)} (| \nabla u |)^{\frac{p^{+}}{q^{-}}}, \, \rho_{q( \cdot)} ( | \nabla u |)^{\frac{p^{+}}{q^{+}}} \} \nonumber \\& \leq  D^{p^{+}} ( c + 1 )^{p^{+}} \rho_{q( \cdot)} (| \nabla u |)^{\frac{p^{+}}{q^{-}}}.
\end{align}
\begin{defn}
The Sobolev space $ W^{1, q( x )} ( M )$ consists of such functions $ u \in L^{q( x )} ( M )$ for which $ \nabla^{k} u \in L^{q( x )} ( M )$ $k = 1, 2, \cdots, n.$ The norm is defined by
$$ ||\, u\,||_{W^{1, q( x )} ( M )} = ||\, u\,||_{L^{q( x )} ( M )} + \sum_{k = 1}^{n} ||\, \nabla^{k} u \,||_{L^{q( x )} ( M )}.$$
The space $ W_{0}^{1, q( x )} ( M )$ is defined as the closure of $ C^{\infty}_{c} ( M ) $ in $ W^{1, q( x )} ( M ),$
with $C_{c}^{\infty} ( M ) $ be the vector space of smooth functions with compact support on $M.$
\end{defn}
\begin{theorem}\label{theo1}
Let $M$ be a compact Riemannian manifold with a smooth boundary or without boundary and $ q( x ), \, p( x ) \in C( \overline{M} ) \cap L^{\infty} ( M ).$ Assume that $$ q( x ) < N , \hspace*{0.5cm} p( x ) < \frac{N\, q( x )}{N - q( x )} \,\, \mbox{for} \,\, x \in \overline{M}.$$
Then, $$ W^{1, q( x )} ( M ) \hookrightarrow L^{p( x )} ( M )$$
is a continuous and compact embedding.
\end{theorem}
\begin{proof}
This proof is based to an idea introduced in \cite{fan2001spaces, guo2015dirichlet}. Let $ f : U ( \subset M ) \longrightarrow \mathbb{R}^{N} $ be an arbitrary local chart on $M$, and $V$ be any open set in $M$, whose closure is compact and is contained in $U$. Choosing a finite subcovering $\{\, V_{\alpha}\, \}_{\alpha = 1, \cdots, k} $ of $M$ such that $V_{\alpha}$ is homeomorphic to the open unit ball $B_{0} ( 1 )$ of $ \mathbb{R}^{N} $ and for any $ \alpha $ the components $ g_{ij}^{\alpha}$ of $g$ in $ ( V_{\alpha}, \, f_{\alpha} )$ satisfy $$ \frac{1}{\epsilon\, \delta_{ij}} \leq g_{ij}^{\alpha} < \epsilon \, \delta_{ij}$$ as bilinear forms, where the constant $\epsilon > 1 $ is given. Let $ \{ \, \pi_{\alpha} \}_{\alpha = 1, \cdots, k} $ be a smooth partition of unity subordinate to the finite covering $\{\, V_{\alpha}\, \}_{\alpha = 1, \cdots, k} $. It is obvious that if $ u \in W^{1, q( x )} ( M ), $ then $ \pi_{\alpha} \, u \in W^{1, q( x )} ( V_{\alpha} ) $ and $ ( f_{\alpha}^{-1} )^{*} ( \pi_{\alpha} u ) \in W^{1, q( f_{\alpha}^{-1} ( x ) )} ( B_{0} ( 1 ) ).$ According to propsition \ref{prop10} and  the Sobolev embeddings Theorem in \cite{fan2001spaces, gaczkowski2016sobolev}, we obtain the continuous and compact embedding
$$ W^{1, q( x )} ( V_{\alpha} ) \hookrightarrow L^{p( x )} ( V_{\alpha} ) \hspace*{0.5cm} \mbox{for each} \,\, \alpha = 1, \cdots, k.$$
Since $ u = \displaystyle\sum_{\alpha = 1}^{k} \pi_{\alpha} u,$ we can conclude that $$ W^{1, q( x )} ( M ) \subset L^{p( x )} ( M ),$$ and the embedding is continuous and compact.
\end{proof}
\begin{prop} \cite{aubin1982nonlinear}
If $( M, g )$ is complete, then  $W^{1, q( x )} ( M ) = W^{1, q( x )}_{0} ( M ).$
\end{prop}

The weighted variable exponent Lebesgue space $L_{\mu ( x )}^{q( x )} ( M )$ is defined as follows:
$$ L_{\mu ( x )}^{q( x )} ( M ) = \{ u: M \rightarrow \mathbb{R} \,\, \mbox{is measurable such that,}  \int_{M} \mu ( x ) | u( x ) |^{q( x )} \,\, dv_{g} ( x ) < + \infty \, \},$$
with the norm
$$ || u ||_{q( x ), \mu ( x )} = \inf \{ \gamma > 0 : \int_{M} \mu ( x ) \, \bigg| \frac{u( x )}{\gamma} \bigg|^{q( x )} \,\, dv_{g} ( x ) \leq 1 \,\}.$$
Moreover, the weighted modular on $L_{\mu ( x )}^{q( x )} ( M )$ is the mapping $ \rho_{q( \cdot ), \mu ( \cdot )} : L_{\mu ( x )}^{q( x )} ( M ) \rightarrow \mathbb{R}$ defined like
$$ \rho_{q( \cdot ), \mu ( \cdot )} ( u ) = \int_{M} \mu ( x ) | u( x )|^{q( x )} \,\, dv_{g} ( x ).$$
\begin{ex} 
As a simple example of $\mu ( x ),$ we can take $ \mu ( x ) = ( 1 + | x | )^{\varepsilon ( x )} $ with $\varepsilon ( \cdot ) \in C_{+} ( \overline{M} ).$
\end{ex}
\begin{prop} \label{proposition2.19}
Let $u$ and $ \{ u_{n} \} \subset L^{q( x )}_{\mu ( x )} ( M ),$ then we have the following results:
\begin{enumerate}
\item[$(1)$] $ || u ||_{q( \cdot), \mu ( \cdot )} < 1$ (resp. $= 1, > 1$) $\Longleftrightarrow \rho_{q( \cdot ), \mu ( \cdot )} ( u ) < 1$  (resp. $= 1, > 1$).
\item[$(2)$] $|| u ||_{q( \cdot ), \mu ( \cdot )} < 1 \Rightarrow || u ||^{q^{+}}_{q( \cdot ), \mu ( \cdot )} \leq \rho_{q( \cdot ), \mu ( \cdot )} ( u ) \leq || u ||^{q^{-}}_{q( \cdot ), \mu ( \cdot )}.$
\item[$(3)$] $|| u ||_{q( \cdot ), \mu ( \cdot )} > 1 \Rightarrow || u ||^{q^{-}}_{q( \cdot ), \mu ( \cdot )} \leq \rho_{q( \cdot ), \mu ( \cdot )} ( u ) \leq || u ||^{q^{+}}_{q( \cdot ), \mu ( \cdot )}.$
\item[$(4)$] $ \lim_{n \rightarrow + \infty} || u_{n} ||_{q( \cdot ), \mu ( x )} = 0 \Longleftrightarrow \lim_{n \rightarrow + \infty} \rho_{q( \cdot ), \mu ( \cdot )} ( u_{n} ) = 0.
$
\item[$(5)$] $ \lim_{n \rightarrow + \infty} || u_{n} ||_{q( \cdot ), \mu ( x )} = \infty \Longleftrightarrow \lim_{n \rightarrow + \infty} \rho_{q( \cdot ), \mu ( \cdot )} ( u_{n} ) = \infty.$
\end{enumerate}
\end{prop}
Note that, the non-negative weighted function $ \mu \in C( \overline{M} )$ satisfy the following hypothesis:\\
$ \mu ( \cdot ) : \overline{M} \longrightarrow \mathbb{R}^{+}_{*}$ such that $ \mu (\cdot ) \in L^{\varepsilon ( x )} ( M )$ with 
\begin{equation}\label{2.3}
\frac{N p( x )}{N p( x ) - q( x ) ( N - p( x ) )} < \varepsilon ( x ) < \frac{p( x )}{p( x ) - q( x )} \,\,\, \mbox{for all} \,\,\, x \in \overline{M}.
\end{equation}
Indeed, since $\mu ( \cdot ) : \overline{M} \longrightarrow \mathbb{R}^{+}_{*},$ then, there exists $ \mu_{0} > 0, $ and for all $ x \in M, $ we have that $ \mu ( x ) > \mu_{0}.$
\begin{theorem} \label{theoremK}
Let $M$ be a compact Riemannian manifold with a smooth boundary or without boundary and $ p( x ), \, q( x ) \in C( \overline{M} ) \cap L^{\infty} ( M ).$ Assume that the assumption \eqref{2.3} is true. Then, the embedding $$ W^{1, q( x )} ( M ) \hookrightarrow L^{q( x )}_{\mu ( x )} ( M ),$$ is compact 
\end{theorem}
\begin{proof}
Let $ \theta  ( x ) = \frac{\varepsilon ( x )}{\varepsilon ( x ) - 1} q( x ) = \hat{\varepsilon} ( x ) q( x ),$ where $ \frac{1}{\varepsilon ( x )} + \frac{1}{\hat{\varepsilon} ( x )} = 1.$ From, \eqref{2.3}, we deduce that $\theta ( x ) < p^{*} ( x )$ for all $ x \in \overline{M},$ which implies by Theorem \ref{theo1}, that $W^{1, q( x )} ( M ) \hookrightarrow L^{\theta ( x )} ( M ).$ Hence, we have that $ | u |^{q( x )} \in L^{\theta ( x )} ( M )$ for any $u \in W^{1, q( x )} ( M )$. Now, using the Hölder inequality, we get
\begin{equation} \label{2.4}
\rho_{q( \cdot ), \mu ( \cdot )} ( u ) \leq r_{q} \,.\, || \mu ( x ) ||_{\varepsilon ( x )} || \, | u |^{q( x )} \, ||_{\hat{\varepsilon} ( x )} < + \infty.
\end{equation}
It follows that $ u \in L^{q( x )}_{\mu ( x )} ( M ),$ that is $$ W^{1, q( x )} ( M ) \hookrightarrow L^{q( x )}_{\mu ( x )} ( M ).$$
Next, we prove that this embedding is compact. For that, we consider $\{ u_{n} \} \subset W^{1, q( x )} ( M )$ such that $ u_{n} \rightharpoonup 0 $ weakly in $W^{1, q( x )} ( M )$ and since $  W^{1, q( x )} ( M ) \hookrightarrow  \hookrightarrow L^{\theta ( x )} ( M ),$ we obtain that 
$$ u_{n} \longrightarrow 0 \,\,\, \mbox{in} \,\,\, L^{\theta ( x )} ( M ).$$
Then, it follows that $ || \,  | u |^{q( x )} \,||_{\hat{\varepsilon} ( x )} \rightarrow 0 $ as $ n \rightarrow + \infty.$ By Hölder inequality and \eqref{2.4}, we have $$ \rho_{q( \cdot ), \mu ( \cdot )} ( u_{n} ) \longrightarrow 0.$$
From proposition \ref{proposition2.19}, result ( 4 ), we deduce that 
\begin{equation} \label{2.5}
|| u _{n} ||_{q(\cdot), \mu ( \cdot )} \longrightarrow 0 \,\,\, \mbox{as} \,\,\, n \rightarrow + \infty.
\end{equation}
Hence, the embedding $W^{1, q( x )} ( M ) \hookrightarrow L^{q( x )}_{\mu ( x )} ( M )$ is compact.
\end{proof}

\section{Nehari Manifold Analysis for $( \mathcal{P} )$} \label{section3}
In this section, we note by $D( M )$ the space of $C_{c}^{\infty}$ functions with compact support in $M$.
\begin{defn}\label{def 3.1}
$u \in W^{1, q( x )}_{0} ( M )$ is said to be a weak solution of the problem $ ( \mathcal{P} )$ if for every $\phi \in D( M ) $ we have
\begin{align*}
&\int_{M} \big( \, |\, \nabla u( x ) \,|^{p( x ) - 2} + \mu ( x ) |\, \nabla u( x ) \,|^{q( x ) - 2} \big) \,.\, g( \nabla u( x ), \, \nabla \phi ( x ) ) \,\, dv_{g} ( x ) \\&= \lambda \int_{M} | \, u( x ) |^{q( x ) - 2} \,.\, u( x ) \, \phi ( x ) \,\, dv_{g} ( x ) - \int_{M} |\, u( x )\, |^{p( x ) - 2} \,.\, u( x ) \,.\, \phi ( x ) \,\, dv_{g} ( x ) \\& \hspace*{0.3cm}+ \int_{M} f( x, u( x ) ) \,.\, \phi ( x ) \,\, dv_{g} ( x ).
\end{align*}
\end{defn}
The variable exponents $p, q \in C( \overline{M} )$, are assumed to satisfy the following assumption:
\begin{equation} \label{1.1}
1 < q^{-} \leq q^{+} < p^{-} \leq p^{+} < N.
\end{equation}
Then, we have
\begin{equation}\label{1994}
\frac{p^{+}}{q^{+} - q^{-}} < \frac{p^{+} - q^{+}}{p^{+} - q^{-}} - \frac{( q^{+} - q^{-} )\,.\,( p^{+} - q^{+} )}{( p^{+} - q^{-} )\,.\, ( p^{-} - q^{-} )}.
\end{equation}
We suppose $\frac{p^{-}}{q^{+}} \leq 1 + \frac{1}{N},$ and the function $ \mu : \overline{M} \rightarrow \mathbb{R}_{*}^{+}$ is Lipschitz continuous.\\

Let us consider the energy functional $ J_{\lambda} : W^{1, q( x )}_{0} ( M ) \longrightarrow \mathbb{R}$ associated to problem $ ( \mathcal{P} )$ which is defined by
\begin{align*}
J_{\lambda} ( u ) &= \int_{M} \frac{1}{p( x )} \, |\, \nabla u( x )\,|^{p( x )} \,\, dv_{g} ( x ) + \int_{M} \frac{\mu ( x )}{q( x )} \, |\, \nabla u( x )\,|^{q( x )} \,\, dv_{g} ( x ) \\& \hspace*{0.3cm}- \int_{M} \frac{\lambda}{q( x )} \, |\, u( x ) \,|^{q( x )} \,\, dv_{g} ( x ) + \int_{M} \frac{1}{p( x )} \, |\, u( x ) \,|^{p( x )} \,\, dv_{g} ( x ) \\& \hspace*{0.3cm} - \int_{M} F( x, u( x ) ) \,\, dv_{g} ( x ).
\end{align*}
And, for any $ u \in W^{1, q( x )}_{0} ( M )$ with $ ||\, u\,||_{W^{1, q( x )}_{0} ( M )} > 1, $ we have by $ ( f_{1} ), ( f_{2} ),$ \eqref{Said}, Proposition \ref{prop61} and Poincaré inequality that 
\begin{align*}
J_{\lambda} ( u ) & = \int_{M} \frac{1}{p( x )} \, |\, \nabla u( x )\,|^{p( x )} \,\, dv_{g} ( x ) + \int_{M} \frac{\mu ( x )}{q( x )} \, |\, \nabla u( x )\,|^{q( x )} \,\, dv_{g} ( x ) \\& \hspace*{0.3cm}- \int_{M} \frac{\lambda}{q( x )} \, |\, u( x ) \,|^{q( x )} \,\, dv_{g} ( x ) + \int_{M} \frac{1}{p( x )} \, |\, u( x ) \,|^{p( x )} \,\, dv_{g} ( x ) \\& \hspace*{0.3cm} - \int_{M} F( x, u( x ) ) \,\, dv_{g} ( x ) \\&\geq  \frac{1}{p^{+}} \int_{M} |\, \nabla u( x )\,|^{p( x )} \,\, dv_{g} ( x ) + \frac{\mu_{0}}{D^{p^{+}} ( c + 1 )^{p^{+}} q^{+}} \int_{M} |\, u( x )\,|^{p( x )} \,\,  dv_{g} ( x ) \\& \hspace*{0.3cm}- \frac{\lambda}{q^{-}} \int_{M} |\, u( x )|^{q( x )} \,\, dv_{g} ( x ) + \frac{1}{p^{+}} \int_{M} |\, u( x )\,|^{p( x )} \,\, dv_{g} ( x ) \\& \hspace*{0.3cm}- \frac{1}{\beta} \int_{M} f( x, u( x ) ) \,.\, u( x ) \,\, dv_{g} ( x ) \\& \geq \frac{1}{c p^{+}} \rho_{p( \cdot )} ( u ) + \frac{\mu_{0}}{D^{p^{+}} ( c + 1 )^{p^{+}}  q^{+}} \rho_{p( \cdot)} ( u ) - \frac{\lambda}{q^{-}} \rho_{q( \cdot )} ( u ) + \frac{1}{p^{+}} \rho_{p( \cdot )} ( u ) - \frac{1}{p^{+}} \rho_{q( \cdot )} ( u ) \\& \hspace*{0.3cm} ( \mbox{since} \,\, \beta > p^{+} \,\, \mbox{from}\,\, ( f_{1} ), \,\ \mbox{and} \,\, c \,\, \mbox{is the Poincaré constant} ). 
\end{align*}
According to the proposition \ref{prop13}, we have that 
\begin{align*}
J_{\lambda} ( u ) &\geq  \bigg( \frac{1}{c p^{+}} + \frac{1}{p^{+}} + \frac{\mu_{0}}{D^{p^{+}} ( c + 1 )^{p^{+}} q^{+}} \bigg) \, \rho_{p( \cdot )} ( u ) - \frac{\lambda}{q^{-}} \rho_{q( \cdot )} ( u ) - \frac{1}{p^{+}} \rho_{q( \cdot )} (u) \\& \geq \bigg( \frac{1}{c p^{+}} + \frac{1}{p^{+}} + \frac{\mu_{0}}{D^{p^{+}} ( c + 1 )^{p^{+}} q^{+}} \bigg) \, || u ||^{p^{-}}_{W^{1, q( x )}_{0} ( M )} - \bigg( \frac{\lambda}{q^{-}} + \frac{1}{p^{+}} \bigg) \, || u ||^{q^{+}}_{W^{1, q( x )}_{0} ( M )}.
\end{align*}
From \eqref{1.1}, we have that $ J_{\lambda} $ is not bounded below on the whole space $ W^{1, q( x )}_{0} ( M ),$ but it is bounded above on an appropriate subset of $W^{1, q( x )}_{0} ( M )$ which is the Nehari manifold associated to $J_{\lambda}$ defined by
$$ \mathcal{N}_{\lambda} = \{\, u \in W_{0}^{1, q( x )} ( M ) \backslash \{ 0 \} : \langle\, J_{\lambda}^{'} ( u ) , \, u \, \rangle = 0\, \}.$$
Indeed, if we take for example $X$ a Banach space, and $J$ the Euler (energy) functional associated with a variational problem on $X.$ If $J$ is bounded above and has a minimizer, then, this minimizer is a critical point of $J.$ Therefore, it is a weak solution of the variational problem. However, in many problems, $J$ is not bounded on the whole space $X,$ but is bounded on an appropriate subset of $X,$ which is the case of our problem. \\
So, it is clear that the critical points of the functional $J_{\lambda}$ must lie on $\mathcal{N}_{\lambda}$ and local minimizers on $\mathcal{N}_{\lambda}$ are usually critical points of $J_{\lambda}$. Thus, $ u\in \mathcal{N}_{\lambda}$ if and only if
\begin{align}\label{3.1}
\langle  J_{\lambda}^{'} ( u ) , \, u \rangle  & = \int_{M} |\, \nabla u( x )\,|^{p( x )} \,\, dv_{g} ( x ) + \int_{M} \mu ( x ) \,|\, \nabla u( x ) \,|^{q( x )} \,\, dv_{g} ( x ) \nonumber \\&- \lambda \, \int_{M} |\, u( x )\, |^{q( x )} \,\, dv_{g} ( x ) + \int_{M} |\, u( x )\,|^{p( x )} \,\, dv_{g} ( x ) \nonumber \\&- \int_{M} f( x, u( x ) ) \,.\, u( x ) \,\, dv_{g} ( x ) = 0.
\end{align}
Hence, $\mathcal{N}_{\lambda}$ contains every nontrivial weak solution of problem $( \mathcal{P} ) $ (see definition \ref{def 3.1}). Moreover, we have the following result

\begin{lemma}\label{lem 3.2}
Under assumptions $( f_{1} ) - ( f_{3} ).$ The energy functional $J_{\lambda}$ is coercive and bounded below on $ W^{1, q( x )}_{0} ( M )$.
\end{lemma}
\begin{proof}
Let $ u \in \mathcal{N}_{\lambda}$ and $ ||\, u\,|| > 1,$ where $||\, . \,||$ is the induced norm of $W_{0}^{1, q( x )} (\Omega ) \backslash \{ 0 \}.$ Then, by \eqref{3.1}, \eqref{1.1}, \eqref{Said}, $( f_{1} ),\, ( f_{3} ),$  propositions \ref{prop61} and \ref{prop13}, we have
\begin{align*}
&J_{\lambda} ( u ) \geq \frac{1}{p^{+}} \int_{M} |\, \nabla u ( x )\,|^{p( x )} \,\, dv_{g} ( x )+ \frac{1}{q^{+}} \, \int_{M} \mu ( x ) |\, \nabla u( x ) \,|^{q( x )} \,\, dv_{g} ( x ) \\& \hspace*{0.3cm}- \frac{\lambda}{q^{-}} \, \int_{M} |\, u( x ) \,|^{q( x )} \,\, dv_{g} ( x ) + \frac{1}{p^{+}} \int_{M} | u( x ) |^{p( x )} \,\, dv_{g} ( x ) \\& \hspace*{0.3cm} - \int_{M} F( x, u( x ) ) \,\, dv_{g} ( x ) \\& = \frac{1}{p^{+}} \int_{M} \nabla u( x ) |^{p( x )} \,\, dv_{g} ( x ) + \frac{1}{q^{+}} \int_{M} \mu ( x ) | \nabla u( x ) |^{q( x )} \,\, dv_{g} ( x )\\& \hspace*{0.3cm} - \frac{\lambda}{q^{-}} \int_{M} | u( x ) |^{q( x )} \,\, dv_{g} ( x ) + \frac{1}{p^{+}} \bigg[ - \int_{M} | \nabla u( x ) |^{p( x )} \,\, dv_{g} ( x ) \\& \hspace*{0.3cm}- \int_{M} \mu ( x ) | \nabla u( x ) |^{q( x )} \,\, dv_{g} ( x ) + \lambda \int_{M} | u( x ) |^{q( x )} \,\, dv_{g} ( x ) \\& \hspace*{0.3cm}+ \int_{M} f( x, u( x ) )\,.\, u( x ) \,\, dv_{g} ( x ) \bigg] - \int_{M} F( x, u( x ) ) \,\, dv_{g} ( x ) \\& \geq \mu_{0} \bigg( \frac{1}{q^{+}} - \frac{1}{p^{+}} \bigg) \int_{M} | \nabla u( x )|^{q( x )} \,\, dv_{g} ( x ) + \lambda \bigg( \frac{1}{p^{+}} - \frac{1}{q^{-}} \bigg) \int_{M} | u ( x )|^{q( x )} \,\,dv_{g} ( x ) \\& \hspace*{0.3cm} + \frac{1}{\beta} \int_{M} f( x, u( x ) ) \,.\, u( x ) \,\, dv_{g} ( x ) + \int_{M} F( x, u( x ) ) \,\, dv_{g} ( x ) \\& \hspace*{0.3cm} \mbox{(since $\beta > p^{+},$ then $\frac{1}{p^{+}} > \frac{1}{\beta} $ and by $( f_{1} )$ we get the following inequality)} \\& \geq \frac{\mu_{0}}{D^{p^{+}} ( c + 1 )^{p^{+}}} \bigg(\frac{p^{+} - q^{+}}{p^{+} q^{+}}\bigg) \rho_{p( \cdot )} ( u ) + \lambda \bigg( \frac{q^{-} - p^{+}}{p^{+}q^{-}} \bigg) \rho_{q( \cdot )} ( u ) \,\, \mbox{( from Proposition \ref{prop61} )} \\& \geq \frac{\mu_{0}}{D^{p^{+}} ( c + 1 )^{p^{+}} } \bigg( \frac{p^{+} - q^{+}}{p^{+} q^{+}} \bigg) || u ||^{p^{-}} + \lambda \bigg( \frac{q^{-} - p^{+}}{p^{+}q^{-}} \bigg) || u ||^{q^{+}} \,\, \mbox{( from Proposition \ref{prop13} )} 
\end{align*}
As $ p^{-} > q^{+},$ then $ J_{\lambda} ( u ) \longrightarrow + \infty $ as $ ||\, u\,|| \rightarrow \infty.$ It follows that $ J_{\lambda} $ is coercive and bounded below on $ \mathcal{N}_{\lambda}.$
\end{proof}
Next, we consider the functional $ \psi : \mathcal{N}_{\lambda} \longrightarrow \mathbb{R} $ defined by $$ \psi_{\lambda} ( u ) = \,\langle  J^{'}_{\lambda} ( u ), \, u \rangle  \,\, \mbox{for all} \,\, u \in \mathcal{N}_{\lambda}.$$
Hence, it is natural to split $ \mathcal{N}_{\lambda}$ into three part: the first set corresponding to local minima, the second set corresponding to local maxima, and the third one corresponding to points of inflection which defined respectively as follows
$$ \mathcal{N}_{\lambda}^{+} = \{\, u \in \mathcal{N}_{\lambda} : \, \langle  \psi^{'}_{\lambda} ( u ), \, u \rangle  > 0 \, \},$$
$$ \mathcal{N}_{\lambda}^{-} = \{\, u \in \mathcal{N}_{\lambda} : \, \langle  \psi^{'}_{\lambda} ( u ), \, u \rangle  < 0 \, \},$$
$$ \mathcal{N}_{\lambda}^{0} = \{\, u \in \mathcal{N}_{\lambda} : \, \langle  \psi^{'}_{\lambda} ( u ), \, u \rangle  = 0 \, \}.$$
\begin{lemma}\label{lem 3.3}
Under assumptions $( f_{1} ) - ( f_{3} )$. There exists $ \lambda^{*} > 0 $ such that for any $ \lambda \in ( 0, \, \lambda^{*} )$ we have $ \mathcal{N}_{\lambda}^{0} = \emptyset. $
\end{lemma}
\begin{proof}
Suppose otherwise, that is $ \mathcal{N}_{\lambda}^{0} \neq \emptyset  $ for all $ \lambda \in \mathbb{R} \backslash \{ 0\}.$ Let $ u \in \mathcal{N}_{\lambda}^{0} $ such that $ ||\, u\,|| > 1.$ Then by \eqref{3.1}, \eqref{1.1}, $( f_{1} )$ and the definition of $\mathcal{N}_{\lambda}^{0}, $ we have
\begin{align*}
0 = \langle \psi^{'}_{\lambda} ( u ), \, u \rangle & \geq \, \,p^{-} \int_{M} |\, \nabla u( x )\,|^{p( x )} \,\, dv_{g} ( x )  +  q^{-} \int_{M} \mu ( x ) \,|\, \nabla u ( x ) \,|^{q( x )} \,\, dv_{g} ( x )\\& \hspace*{0.2cm} - q^{+} \bigg[ \int_{M} |\, \nabla u( x ) \,|^{p( x )} \,\, dv_{g} ( x ) +  \int_{M} \mu ( x )\, |\, \nabla u( x )\,|^{q( x )} \,\, dv_{g} ( x ) \\ & \hspace*{0.2cm}+ \int_{M} |\, u( x )\,|^{p( x )} \,\, dv_{g} ( x ) - \int_{M} f( x, u( x ) ) \,.\, u( x ) \,\, dv_{g} ( x ) \, \bigg] \\& \hspace*{0.2cm}+ p^{-} \int_{M} |\, u( x )\,|^{p( x )} \,\, dv_{g} ( x ) - \int_{M} F( x, u( x ) ) \,\, dv_{g} ( x ) \\& \geq ( p^{-} - q^{+} ) \int_{M} | \nabla u( x ) |^{p( x )} \,\, dv_{g} ( x ) \\&\hspace*{0.2cm}+ ( q^{-} - q^{+} ) \int_{M} \mu ( x ) | \nabla u( x ) |^{q( x )} \,\, dv_{g} ( x ) \\& \hspace*{0.2cm}+ ( p^{-} - q^{+} ) \int_{M} | u( x ) |^{p( x )} \,\, dv_{g} ( x ) \\& \hspace*{0.2cm} + q^{+} \int_{M} f( x, u( x ) ) \,.\, u( x ) \,\, dv_{g} ( x ) - \int_{M} F( x, u( x ) ) \,\, dv_{g} ( x ) \\& \hspace*{0.2cm} \mbox{(since $ \beta > p^{+} > q^{+} $, and by \eqref{1.1} we have $ q^{+} > \frac{1}{q^{+}} > \frac{1}{\beta}.$ Then,} \\& \hspace*{0.4cm}\mbox{by $( f_{1} )$ we get the following inequality)}\\& \geq ( q^{-} - q^{+} ) \int_{M} \mu ( x ) | \nabla u( x ) |^{q( x )} \,\, dv_{g} ( x ) \\& \hspace*{0.2cm}+ ( p^{-} - q^{+} ) \int_{M} | u( x ) |^{p( x )} \,\, dv_{g} ( x ).
\end{align*}
Then, $$0 \geq ( p^{-} - q^{+} ) \int_{M} | u( x ) |^{p( x )} \,\, dv_{g} ( x ) + ( q^{-} - q^{+} ) \int_{M} \mu ( x ) | \nabla u( x ) |^{q( x )} \,\, dv_{g} ( x ).$$
Therefore, by Propositions \ref{prop13}, \ref{proposition2.19} and Theorem \ref{theoremK}, we have 
$$ 0 \geq ( p^{-} - q^{+} ) || u ||^{p^{-}} + c_{1} ( q^{-} - q^{+} ) || u ||^{q^{+}}, $$
where $c_{1}$ being the constant of the embedding Theorem \ref{theoremK}.\\
Hence,
\begin{equation}\label{3.2}
||\, u\,|| \leq \bigg( \frac{c_{1} ( q^{+} - q^{-} )}{p^{-} - q^{+}}\, \bigg)^{\frac{1}{p^{-} - q^{+}}}.
\end{equation}
Analogously:
\begin{align*}
0 = \langle \psi^{'}_{\lambda} ( u ), u \rangle &\leq \, p^{+} \int_{M} |\, \nabla u( x )\,|^{p( x )} \,\, dv_{g} ( x ) + q^{+} \int_{M}  \mu ( x ) \,|\, \nabla u( x ) \,|^{q( x )} \,\, dv_{g} ( x ) \\& \hspace*{0.2cm}- \lambda \, q^{-} \int_{M} |\, u( x ) \,|^{q( x )} \,\, dv_{g} ( x ) + p^{+} \, \bigg[ \, - \int_{M} |\, \nabla u( x ) \,|^{p( x )} \,\, dv_{g} ( x ) \\& \hspace*{0.2cm}- \int_{M}\mu ( x )\, |\, \nabla u( x )\,|^{q( x )} \,\, dv_{g} ( x ) + \lambda \int_{M} |\, u( x )\,|^{q( x )} \,\, dv_{g} ( x ) \\& \hspace*{0.2cm}+ \int_{M} f( x, u( x ) ) \,.\, u( x ) \,\, dv_{g} ( x ) \, \bigg] - \int_{M} F( x, u( x ) ) \,\, dv_{g} ( x )\\& \leq ( q^{+} - p^{+} ) \int_{M}\mu ( x ) \, |\, \nabla u ( x ) \, |^{q( x )} \,\, dv_{g} ( x ) \\& \hspace*{0.2cm}+ \lambda ( p^{+} - q^{-} ) \int_{M} |\, u( x )\,|^{q( x )} \,\, dv_{g} ( x ) + p^{+} \int_{M} f( x, u( x ) )\,.\, u( x ) \,\, dv_{g} ( x ) \\& \leq ( q^{+} - p^{+} ) \int_{M}\mu ( x ) \, |\, \nabla u ( x ) \, |^{q( x )} \,\, dv_{g} ( x ) \\& \hspace*{0.2cm}+ \lambda ( p^{+} - q^{-} ) \int_{M} |\, u( x )\,|^{q( x )} \,\, dv_{g} ( x ) + p^{+} \int_{M} | u( x ) |^{q( x )} \,\, dv_{g} ( x ).
\end{align*}
Then, 
\begin{align*}
\mu_{0} ( p^{+} - q^{+} ) \int_{M} | \nabla u( x ) |^{q( x )} \,\, dv_{g} ( x ) &\leq (p^{+} - q^{+} ) \int_{M} \mu ( x ) | \nabla u( x ) |^{q( x )} \,\, dv_{g} ( x ) \\& \leq \big[ \lambda ( p^{+} - q^{-} ) + p^{+} \big] \int_{M} | u( x ) |^{q( x )} \,\, dv_{g} ( x ).
\end{align*}
By \eqref{Said} and proposition \ref{prop61} we deduce that
$$ \frac{\mu_{0}}{D^{p^{+}} ( c + 1 )^{p^{+}}} \, ( p^{+} - q^{+} ) \, ||\, u\,||^{p^{-}} \leq \big[ \lambda \, ( p^{+} - q^{-} ) + p^{+} \big] \, ||\, u\,||^{q^{+}}. $$
Thus,
\begin{equation}\label{3.3}
||\, u\,|| \leq \bigg( \frac{D^{p^{+}} ( c + 1 )^{p^{+}} \big[ \lambda \, ( p^{+} - q^{-} ) + p^{+} \big]}{\mu_{0} \, ( p^{+} - q^{+} )} \, \bigg)^{\frac{1}{p^{-} - q^{+}}}.
\end{equation}
For $\lambda$ sufficiently small $ \big( \, \lambda < \frac{2 \,\mu_{0} \, ( p^{+} - q^{+} )}{D^{p^{+}} ( c + 1 )^{p^{+}} ( p^{+} - q^{-} )} \, -\, \frac{\mu_{0} \,c_{1} \, ( q^{+} - q^{-} ) ( p^{+} - q^{+} )}{D^{p^{+}} ( c + 1 )^{p^{+}} ( p^{+} - q^{-}) ( p^{-} - q^{-} )} - \frac{p^{+}}{p^{+} - q^{-}}\, \big) $, if we combining \eqref{1994}, \eqref{3.2} and \eqref{3.3} we find $ ||\, u\,|| < 1$ for $\mu_{0}$ sufficiently large, which contradicts our assumption. Consequently, we can conclude that there exists $ \lambda^{*} > 0$ such that $ \mathcal{N}_{\lambda}^{0} = \emptyset $ for any $ \lambda \in ( 0, \, \lambda^{*} ).$
\end{proof}
\begin{rem}
As a consequence of Lemma \ref{lem 3.3}, for $ 0 < \lambda < \lambda^{*}, $ we can write $ \mathcal{N}_{\lambda} = \mathcal{N}_{\lambda}^{+} \cup \mathcal{N}_{\lambda}^{-},$ and we define
$$ \theta_{\lambda}^{+} = \inf_{u \in \mathcal{N}_{\lambda}^{+}} J_{\lambda} ( u ), \hspace*{1cm} \theta_{\lambda}^{-} = \inf_{u \in \mathcal{N}_{\lambda}^{-}} J_{\lambda} ( u ).$$
\end{rem}
\begin{lemma}\label{lem3.5}
Suppose that $( f_{1} ) - ( f_{3} ) $ are true. If $ 0 < \lambda < \lambda^{*}$ with $\lambda^{*} > 0,$ then for all $ u \in \mathcal{N}_{\lambda}^{+} $ we have $ J_{\lambda} ( u ) < 0.$
\end{lemma}
\begin{proof}
Suppose $ u \in \mathcal{N}_{\lambda}^{+},$ from the definition of $ J_{\lambda},$ we have
\begin{align}\label{3.4}
J_{\lambda} ( u ) &\leq \frac{1}{p^{-}} \int_{M} |\, \nabla u( x ) \,|^{p( x )} \,\, dv_{g} ( x ) + \frac{1}{q^{-}} \int_{M} \mu ( x ) \,|\, \nabla u( x ) \,|^{q( x )} \,\, dv_{g} ( x ) \nonumber \\& \hspace*{0.2cm}- \frac{\lambda}{q^{+}} \int_{M} |\, u( x ) \,|^{q( x )} \,\, dv_{g} ( x ) + \frac{1}{p^{-}} \int_{M} |\, u( x )\,|^{p( x )} \,\, dv_{g} ( x ) \nonumber \\& \hspace*{0.2cm}- \int_{M} F ( x, u( x ) ) \,\, dv_{g} ( x ),
\end{align}
from \eqref{3.1} and \eqref{3.4} we have
\begin{align*}
J_{\lambda} ( u ) \leq&  \frac{1}{p^{-}} \int_{M} |\, \nabla u( x ) \,|^{p( x )} \,\, dv_{g} ( x ) + \frac{1}{q^{-}} \int_{M} \mu ( x ) \,|\, \nabla u( x )\,|^{p( x )} \,\, dv_{g} ( x )\\& - \frac{1}{q^{+}} \, \bigg[ \, \int_{M} |\, \nabla u( x )\,|^{p( x )} \,\, dv_{g} ( x ) + \int_{M} \mu ( x )\,|\, \nabla u( x ) \,|^{q ( x )} \,\, dv_{g} ( x ) \\& + \int_{M} |\, u( x ) \,|^{p( x )} \,\, dv_{g} ( x ) - \int_{M} f( x, u( x ) )\,.\, u( x ) \,\, dv_{g} ( x ) \, \bigg] \\&+ \frac{1}{p^{-}} \int_{M} |\, u( x ) \,|^{p( x )} \,\, dv_{g} ( x ) - \int_{M} F( x, u( x ) ) \,\, dv_{g} ( x ) \\& \leq \big( \frac{1}{p^{-}} - \frac{1}{q^{+}} \big) \int_{M} | \nabla u( x ) |^{p( x )} \,\, dv_{g} ( x ) + \big( \frac{1}{q^{-}} - \frac{1}{q^{+}} \big) \int_{M} \mu ( x )\, | \nabla u( x ) |^{q( x )} \,\, dv_{g} ( x ) \\& \hspace*{0.2cm} + \big( \frac{1}{p^{-}} - \frac{1}{q^{+}} \big) \int_{M} | u( x ) |^{p( x )} \,\, dv_{g} ( x ) + \frac{1}{q^{+}} \int_{M} f( x, u( x ) )\,.\, u( x ) \,\, dv_{g} (x)\\& \hspace*{0.2cm} - \int_{M} F( x, u( x ) )\,\, dv_{g} ( x ),
\end{align*}
using Poincaré inequality and Theorem \ref{theoremK}, we find that:
\begin{equation}\label{O1}
J_{\lambda} ( u ) \leq - \bigg( \frac{p^{-} - q^{+}}{p^{-} q^{+}} \bigg) \bigg( \frac{1}{c} + 1 \bigg) || u ||^{p^{-}} + \bigg[ c_{1} \bigg(\frac{q^{+} - q^{-}}{q^{-} q^{+}} \bigg) + \frac{1}{q^{+}} \bigg] \, || u ||^{q^{+}},
\end{equation}
where $c_{1}$ being the constant of the embedding Theorem \ref{theoremK}.\\
As $ u \in \mathcal{N}_{\lambda}^{+}$ we have
\begin{align}\label{O2}
&p^{+} \int_{M} | \nabla u( x ) |^{p( x )} \,\, dv_{g} ( x ) + q^{+} \int_{M} \mu ( x ) |\nabla u( x )|^{q( x )} \,\, dv_{g} ( x )\nonumber\\& - \lambda q^{-} \int_{M} | u( x ) |^{q( x )} \,\, dv_{g} ( x ) + p^{+} \int_{M} | u( x ) |^{p( x )} \,\, dv_{g} ( x ) - \int_{M} F( x, u( x ) ) \,\, dv_{g} ( x ) > 0, 
\end{align}
we multiply \eqref{3.1} by $( -p^{+} ),$ we obtain 
\begin{align}\label{O3}
&- p^{+} \int_{M} | \nabla u( x ) |^{p( x )} \,\, dv_{g} ( x ) - p^{+} \int_{M} \mu ( x ) | \nabla u( x ) |^{q( x )} \,\, dv_{g} ( x )\nonumber\\& + \lambda p^{+} \int_{M} | u( x ) |^{q( x )} \,\, dv_{g} ( x ) - p^{+} \int_{M} | u( x ) |^{p( x )} \,\,dv_{g} ( x )\nonumber \\&+ p^{+} \int_{M} f( x, u( x ) ) \,.\, u( x ) \,\, dv_{g} ( x ) = 0, 
\end{align}
we add \eqref{O2} to \eqref{O3}, then according to the fact that $\beta > p^{+}$ and by \eqref{1.1} we have $ p^{+} > \frac{1}{p^{+}} > \frac{1}{\beta}$. Therefore, we obtain that 
$$
 ( q^{+} - p^{+} ) \int_{M} \mu( x ) | \nabla u( x ) |^{q( x )} \,\, dv_{g} ( x ) + \lambda ( p^{+} - q^{-} ) \int_{M} | u( x ) |^{q( x )} \,\, dv_{g} ( x )  > 0.
$$
Then, by \eqref{Said} and Proposition \ref{prop61} we get
\begin{align*}
\lambda ( p^{+} - q^{-} ) \int_{M} | u( x ) |^{q( x )} \,\, dv_{g} ( x ) & > ( p^{+} - q^{+} ) \int_{M} \mu( x ) \, | \nabla u( x ) |^{q( x )} \,\, dv_{g} ( x ) \\& > \frac{\mu_{0} ( p^{+} - q^{+} )}{D^{p^{+}} ( c + 1 )^{p^{+}}} \int_{M} | u( x ) |^{p( x )} \,\, dv_{g} ( x ).
\end{align*}
Hence,
$$ || u ||^{p^{-}} < \frac{\lambda \, D^{p^{+}} ( c + 1 )^{p^{+}} ( p^{+} - q^{-} )}{\mu_{0} ( p^{+} - q^{+} )} \, || u ||^{q^{+}}. $$
According to \eqref{O1}, we get
\begin{align*}
J_{\lambda} ( u ) \leq \bigg[ &- \bigg( \frac{p^{-} - q^{+}}{p^{-} q^{+}} \bigg) \bigg( \frac{1}{c} + 1 \bigg) \,.\,\frac{\lambda \, D^{p^{+}} ( c + 1 )^{p^{+}} ( p^{+} - q^{-} )}{\mu_{0} ( p^{+} - q^{+} )}\\& + c_{1} \bigg(\frac{q^{+} - q^{-}}{q^{-} q^{+}} \bigg) + \frac{1}{q^{+}} \bigg]  || u ||^{q^{+}}.
\end{align*}
Finally, for $\lambda$ sufficiently large, we deduce that $ \theta_{\lambda}^{+} = \inf_{u \in \mathcal{N}_{\lambda}^{+}} J_{\lambda} ( u ) < 0.$
\end{proof}
\begin{lemma}\label{lem3.6}
Under assumptions $ ( f_{1} ) - ( f_{3} )$. If $ 0 < \lambda < \lambda^{**},$ then for all $ u \in \mathcal{N}_{\lambda}^{-} $ we have $ J_{\lambda} ( u ) > 0.$
\end{lemma}
\begin{proof}
Let $ u \in \mathcal{N}_{\lambda}^{-}.$ By \eqref{1.1}, $( f_{1} ),$ \eqref{3.1} and the definition of $J_{\lambda},$ we find that
\begin{align*}
J_{\lambda} ( u ) &\geq \frac{1}{p^{+}} \int_{M} |\, \nabla u( x )\,|^{p( x )} \,\, dv_{g} ( x ) + \frac{1}{q^{+}} \int_{M} \mu ( x )\, |\, \nabla u( x )\,|^{q( x )} \,\, dv_{g} ( x ) \\& \hspace*{0.2cm}- \frac{\lambda}{q^{-}} \int_{M} |\, u( x )\,|^{q( x )} \,\, dv_{g} ( x )  + \frac{1}{p^{+}} \bigg[ \, - \int_{M} |\, \nabla u( x ) \,|^{p( x )} \,\, dv_{g} ( x ) \\& \hspace*{0.2cm}-  \int_{M} \mu ( x ) \,|\, \nabla u( x )\,|^{q( x )} \,\, dv_{g} ( x ) + \lambda \int_{M} |\, u( x ) \,|^{q( x )} \,\, dv_{g} ( x ) \\& \hspace*{0.2cm}+ \int_{M} f( x, u( x ) ) \,.\, u( x ) \,\, dv_{g} ( x ) \, \bigg] - \int_{M} F( x, u( x ) ) \,\, dv_{g} ( x ) \\& \geq \big( \frac{1}{q^{+}} - \frac{1}{p^{+}} \big) \int_{M} \mu ( x )  \,|\, \nabla u( x ) \,|^{q( x )} \,\, dv_{g} ( x ) \\& \hspace*{0.2cm}+ \lambda \, \big( \frac{1}{p^{+}} - \frac{1}{q^{-}} \big) \int_{M} |\, u( x )\,|^{q( x )} \,\, dv_{g} ( x ) + \frac{1}{p^{+}} \int_{M} f( x, u( x ) ) \,.\, u( x ) \,\, dv_{g} ( x ) \\& \hspace*{0.2cm}- \int_{M} F( x, u( x ) ) \,\, dv_{g} ( x ) \\& \geq \big( \frac{1}{q^{+}} - \frac{1}{p^{+}} \big) \int_{M} \mu ( x )  \,|\, \nabla u( x ) \,|^{q( x )} \,\, dv_{g} ( x ) \\& \hspace*{0.2cm}+ \lambda \, \big( \frac{1}{p^{+}} - \frac{1}{q^{-}} \big) \int_{M} |\, u( x )\,|^{q( x )} \,\, dv_{g} ( x ) + \frac{1}{\beta} \int_{M} f( x, u( x ) ) \,.\, u( x ) \,\, dv_{g} ( x ) \\& \hspace*{0.2cm}- \int_{M} F( x, u( x ) ) \,\, dv_{g} ( x ) \\& \,\, \mbox{(since $\beta > p^{+},$ then $\frac{1}{p^{+}} > \frac{1}{\beta} $ and by $( f_{1} )$ we get the following inequality)} \\& \geq \mu_{0} \, \big( \frac{1}{q^{+}} - \frac{1}{p^{+}} \big) \int_{M} |\, \nabla u( x )\,|^{q( x )} \,\, dv_{g} ( x ) \\& \hspace*{0.2cm}+ \lambda \, \big( \frac{1}{p^{+}} - \frac{1}{q^{-}} \big) \int_{M} |\, u( x ) \,|^{q( x )} \,\, dv_{g} ( x ),
\end{align*}
according to \eqref{Said} we deduce that
$$ J_{\lambda} ( u ) \geq \frac{\mu_{0}}{D^{p^{+}} ( c + 1 )^{p^{+}}} \big( \frac{1}{q^{+}} - \frac{1}{p^{+}} \big) \,||\, u\,||^{p^{-}} + \lambda \, \big( \frac{1}{p^{+}} - \frac{1}{q^{-}} \big) \,||\, u\,||^{q^{+}}.$$
Since, $p^{-} > q^{+}$ we have
$$ J_{\lambda} ( u ) \geq \bigg( \frac{\mu_{0}}{D^{p^{+}} ( c + 1 )^{p^{+}}} \,.\, \frac{p^{+} - q^{+}}{q^{+}\,p^{+}} + \lambda \, \frac{q^{-} - p^{+}}{p^{+}\,q^{-}} \, \bigg) \, ||\, u\,||^{p^{-}}.$$
Thus, if we choose $ \lambda < \frac{\mu_{0} \, q^{-} \, ( p^{+} - q^{+} )}{D^{p^{+}} ( c + 1 )^{p^{+}} q^{+} \, ( p^{+} - q^{-} )} = \lambda^{**},$ we deduce that $ J_{\lambda} ( u ) > 0.$ \\
It follows that $ \theta_{\lambda}^{-} = \inf_{u \in \mathcal{N}_{\lambda}^{-}} J_{\lambda} ( u ) > 0.$
\end{proof}
Hence, $ \mathcal{N}_{\lambda} = \mathcal{N}_{\lambda}^{+} \cup \mathcal{N}_{\lambda}^{-}$ and $ \mathcal{N}_{\lambda}^{+} \cap \mathcal{N}_{\lambda}^{-} = \emptyset,$ by above Lemma, we must have $ u \in \mathcal{N}_{\lambda}^{-}.$

\section{Existence of non-negative solutions:} \label{section4}
In this section, we prove the existence of two non-negative solutions of problem $ ( \mathcal{P} ). $ For this, we first show the existence of minimizers in $ \mathcal{N}_{\lambda}^{+} $ and $ \mathcal{N}_{\lambda}^{-}$ for all $ \lambda \in ( 0, \, \bar{\lambda} ),$ where $ \bar{\lambda} = \min \, \{\, \lambda^{*}, \, \lambda^{**} \, \}.$
\begin{theorem}\label{theorem 4.1}
Suppose that $ ( f_{1} ) - ( f_{3} ) $ are true, then for all $ \lambda \in ( 0, \, \lambda^{*} ),$ there exists a minimizer $ u_{0}^{+}$ of $J_{\lambda} ( u )$ on $ \mathcal{N}_{\lambda}^{+}$ such that $ J_{\lambda} ( u_{0}^{+} ) = \theta_{\lambda}^{+}.$
\end{theorem}
\begin{proof}
From Lemma \ref{lem 3.2}, $ J_{\lambda} $ is bounded below on $ \mathcal{N}_{\lambda},$ in particular is bounded below on $ \mathcal{N}_{\lambda}^{+}.$ Then there exists a minimizing sequence $\{\, u_{n}^{+} \,\} \subset \mathcal{N}_{\lambda}^{+}$ such that
$$ \lim_{n \rightarrow + \infty} J_{\lambda} ( u_{n}^{+} ) = \inf_{u \in \mathcal{N}^{+}} J_{\lambda} ( u ) = \theta_{\lambda}^{+} < 0.$$
Since, $J_{\lambda} $ is coercive, $\{\, u_{n}^{+} \,\}$ is bounded in $ W^{1, q( x )}_{0} ( M ).$ Hence we assume that, without loss generality, $ u_{n}^{+} \rightharpoonup u_{0}^{+}$ in $ W^{1, q( x )}_{0} ( M )$ and by the compact embedding ( Theorem \ref{theo1} ) we have
\begin{equation}\label{4.1}
u_{n}^{+} \longrightarrow u_{0}^{+} \,\, \mbox{in} \,\, L^{p( x )} ( M ).
\end{equation}
Now, we shall prove $ u_{n}^{+} \longrightarrow u_{0}^{+} $ in $ W^{1, q( x )}_{0} ( M ).$ Otherwise, let $ u_{n}^{+} \not \to u_{0}^{+}$ in $ W^{1, q( x )}_{0} ( M ).$ Then, we have
\begin{equation}\label{4.2}
\rho_{q( . )} ( u_{0}^{+} ) < \lim_{n \rightarrow + \infty} \inf \, \rho_{q( . )} ( u_{n}^{+} ),
\end{equation}
using \eqref{4.1} we obtain
$$ \int_{M} |\, u_{0}^{+} \,|^{p( x )} \,\, dv_{g} ( x ) = \lim_{n \rightarrow + \infty} \inf \, \int_{M} |\, u_{n}^{+} \,|^{p( x )} \,\, dv_{g} ( x ),$$
since $ \langle J^{'}_{\lambda} ( u_{n}^{+} ), \, u_{n}^{+} \rangle = 0,$ and using the same technique as in Lemma \ref{lem3.6}, we get by \eqref{Said} that
$$ J_{\lambda} ( u_{n}^{+} ) \geq \frac{\mu_{0}}{D^{p^{+}} ( c + 1 )^{p^{+}}} \, \big( \frac{1}{q^{+}} - \frac{1}{p^{+}} \big) \, \rho_{p( . )} ( u_{n}^{+} ) + \lambda \, \big( \frac{1}{p^{+}} - \frac{1}{q^{-}} \big) \, \rho_{q( . )} ( u_{n} ^{+} ).$$
That is
\begin{align*}
\lim_{n \rightarrow + \infty}  J_{\lambda} ( u_{n}^{+} ) \geq & \,\frac{\mu_{0}}{D^{p^{+}} ( c + 1 )^{p^{+}}} \, \big( \frac{1}{q^{+}} - \frac{1}{p^{+}} \big) \, \lim_{n \rightarrow + \infty} \, \rho_{p( . )} ( u_{n}^{+} )\\&  + \lambda \, \big( \frac{1}{p^{+}} - \frac{1}{q^{-}} \big) \, \lim_{n \rightarrow + \infty} \rho_{q( . )} ( u_{n} ^{+} ).
\end{align*}
By \eqref{4.1} and \eqref{4.2}, we have
$$ \theta_{\lambda}^{+} > \frac{\mu_{0}}{D^{p^{+}} ( c + 1 )^{p^{+}}} \, \big( \frac{1}{q^{+}} - \frac{1}{p^{+}} \big) \, ||\, u_{0}^{+} \, ||^{p^{-}} + \lambda \, \big( \frac{1}{p^{+}} - \frac{1}{q^{-}} \big) \, ||\, u_{0}^{+} \,||^{q^{+}},$$
since $ p^{-} > q^{+},$ for $ ||\, u_{0}^{+} \,|| > 1,$ we deduce
$$ \theta_{\lambda}^{+} = \inf_{u \,\in \,\mathcal{N}_{\lambda}^{+}} J_{\lambda} ( u ) > 0,$$
which is a contradiction with Lemma \ref{lem3.5}. Hence
$$ u_{n}^{+} \longrightarrow u_{0}^{+} \,\, \mbox{in} \,\, W_{0}^{1, q( x )} ( M ),$$
and $$ \lim_{n \rightarrow + \infty} J_{\lambda} ( u_{n}^{+} ) = J_{\lambda} ( u_{0}^{+} ) = \theta_{\lambda}^{+}.$$
Consequently, $u_{0}^{+}$ is a minimizer of $ J_{\lambda} $ on $ \mathcal{N}_{\lambda}^{+}.$
\end{proof}
\begin{theorem}\label{theorem4.2}
Suppose that conditions $ ( f_{1} ) - ( f_{3} ) $ are true, and for all $ \lambda \in ( 0, \, \lambda^{**} ),$ there exists a minimizer $u_{0}^{-}$ of $ J_{\lambda}$ on $ \mathcal{N}_{\lambda}^{-} $ such that $ J_{\lambda} ( u_{0}^{-} ) = \theta_{\lambda}^{-}.$
\end{theorem}
\begin{proof}
Since $ J_{\lambda} $ is bounded below on $ \mathcal{N}_{\lambda} $ and so on $ \mathcal{N}_{\lambda}^{-}.$ Then, there exists a minimizing sequence $ \{\, u_{n}^{-}\,\} \subseteq \mathcal{N}_{\lambda}^{-} $ such that $$ \lim_{n \rightarrow + \infty} J_{\lambda} ( u_{n}^{-} ) = \inf_{u \in \mathcal{N}_{\lambda}^{-}} J_{\lambda} ( u ) = \theta_{\lambda}^{-} > 0.$$
As $ J_{\lambda} $ is coercive, $\{\, u_{n}^{-} \,\}$ is bounded in $ W_{0}^{1, q( x )} ( M ).$ Thus without loss of generality, we may assume that, $ u_{n}^{-} \rightharpoonup u_{0}^{-} $ in $ W_{0}^{1, q( x )} ( M )$ and by Theorem \ref{theo1} we have
\begin{equation}\label{4.3}
u_{n}^{-} \longrightarrow u_{0}^{-} \,\, \mbox{in} \,\, L^{p( x )} ( M ).
\end{equation}
On the other hand, if $ u_{0}^{-} \in \mathcal{N}_{\lambda}^{-},$ then there exists a constant $ t > 0 $ such that $ t\, u_{0}^{-} \in \mathcal{N}_{\lambda}^{-}$ and $ J_{\lambda} ( u_{0}^{-} ) \geq J_{\lambda} ( t \, u_{0}^{-} ).$ According to $( f_{1} )$ and the definition of $ \psi_{\lambda}^{'},$ we have
\begin{align*}
\langle \psi_{\lambda}^{'} ( t \, u_{0}^{-} ), \, t \, u_{0}^{-} \rangle & =  \,\int_{M} p( x ) \, |\, \nabla t u_{0}^{-} ( x )\,|^{p( x )} \,\, dv_{g} ( x ) \\& \hspace*{0.2cm}+ q( x ) \int_{M} \mu ( x ) \,|\, \nabla t u_{0}^{-} ( x ) \,|^{q( x )} \,\, dv_{g} ( x ) \\& \hspace*{0.2cm} - \lambda \, q( x ) \int_{M} |\, t u_{0}^{-} ( x ) \,|^{q( x )} \,\, dv_{g} ( x ) \\& \hspace*{0.2cm}+ p( x ) \int_{M} |\, t u_{0}^{-} ( x )\,|^{p( x )} \,\, dv_{g} ( x ) - \int_{M} F( x, \, t u_{0}^{-} ( x ) ) \,\, dv_{g} ( x ) \\& \leq p^{+} \, t^{p^{+}} \int_{M} |\, \nabla u_{0}^{-} \,|^{p( x )} \,\, dv_{g} ( x ) \\& \hspace*{0.2cm} + \, q^{+} \, t^{q^{+}} \int_{M} \mu ( x ) |\, \nabla u_{0}^{-} \,|^{q( x )} \,\, dv_{g} ( x ) \\& \hspace*{0.2cm}- \lambda \, q^{-} \, t^{q^{-}} \int_{M} |\, u_{0} ^{-} ( x )\,|^{q( x )} \,\, dv_{g} ( x ) \\& \hspace*{0.2cm}+ p^{+} \, t^{p^{+}} \int_{M} |\, u_{0}^{-} ( x ) \,|^{p( x )} \,\, dv_{g} ( x ).
\end{align*}
Since $ q^{-} \leq q^{+} < p^{+},$ and by \eqref{Said}, propositions \ref{prop61} and \ref{prop13}, it follows that $ \langle \psi_{\lambda}^{'} ( t \, u_{0}^{-}, \, t\, u_{0}^{-} \rangle < 0.$ Hence by the definition of $ \mathcal{N}_{\lambda}^{-}, \,\, t\, u_{0}^{-} \in \mathcal{N}_{\lambda}^{-}.$ \\
Next, we show that $ u_{n}^{-} \longrightarrow u_{0}^{-} $ in $ W_{0}^{1, q( x )} ( M ).$ Otherwise, suppose $ u_{n}^{-} \not \to u_{0}^{-}$ in $ W_{0}^{1, q( x )} ( M ).$ Then by Fatou's Lemma we have
$$ \int_{M} \mu ( x ) \,|\, \nabla u_{0}^{-} ( x )\, |^{q( x )} \,\, dv_{g} ( x ) \leq \lim_{n \rightarrow + \infty} \int_{M} \mu ( x ) \, |\, \nabla u_{n}^{-} ( x )\,|^{q( x )} \,\, dv_{g} ( x ).$$
By \eqref{4.3} we get
$$ \int_{M} |\, u_{0}^{-} ( x )\,|^{p( x )} \,\, dv_{g} ( x ) \leq \lim_{n \rightarrow + \infty} \int_{M} |\, u_{n}^{-} ( x )\,|^{p( x )} \,\, dv_{g} ( x ),$$
and
$$ \int_{M} |\, \nabla u_{0}^{-} ( x )\,|^{p( x )} \,\, dv_{g} ( x ) \leq \lim_{n \rightarrow + \infty} \int_{M} |\, \nabla u_{n}^{-} ( x )\,|^{p( x )} \,\, dv_{g} ( x ).$$
Then, according the above inequalities and $( f_{1} )$, we obtain
\begin{align*}
J_{\lambda} ( t\, u_{0}^{-} ) & \leq \frac{t^{p^{+}}}{p^{-}} \int_{M} |\, \nabla u_{0}^{-} ( x )\,|^{p( x )} \,\, dv_{g} ( x ) + \frac{t^{q^{+}}}{q^{+}} \int_{M} \mu ( x )\,|\, \nabla u_{0}^{-} ( x )\,|^{q( x )} \,\, dv_{g} ( x ) \\& \hspace*{0.3cm} - \frac{\lambda \, t^{q^{-}}}{q^{+}} \int_{M} |\, u_{0}^{-} ( x ) \,|^{q( x )} \,\, dv_{g} ( x ) + \frac{t^{p^{+}}}{p^{+}} \int_{M} |\, u_{0}^{-} ( x )\,|^{p( x )} \,\, dv_{g} ( x )\\& \hspace*{0.3cm} -\int_{M} F( x, t\, u_{0}^{-} ( x ) ) \,\, dv_{g} ( x ) \\&  \leq \lim_{n \rightarrow + \infty} \bigg[ \, \frac{t^{p^{+}}}{q^{-}} \int_{M} |\, \nabla u_{n}^{-} ( x ) \,|^{p( x )} \,\, dv_{g} ( x ) \\& \hspace*{0.3cm} + \frac{t^{q^{+}}}{q^{+}} \int_{M} \mu ( x ) \, |\, \nabla u_{n}^{-} ( x ) \,|^{q( x )} \,\, dv_{g} ( x ) - \frac{\lambda \, t^{q^{-}}}{q^{+}} \int_{M} |\, u_{n}^{-} ( x ) \,|^{q( x )} \,\, dv_{g} ( x ) \\& \hspace*{0.3cm}  + \frac{t^{p^{+}}}{p^{+}} \int_{M} |\, u_{n}^{-} ( x )\,|^{p( x )} \,\, dv_{g} ( x ) -\int_{M} F( x, t\, u_{0}^{-} ( x ) ) \,\, dv_{g} ( x ) \, \bigg] \\& \leq \lim_{n \rightarrow + \infty} J_{\lambda} ( t \, u_{n}^{-} ) < \lim_{n \rightarrow + \infty} J_{\lambda} ( u_{n}^{-} ) = \inf_{u \,\in \,\mathcal{N}_{\lambda}^{-}} J ( u ) = \theta_{\lambda}^{-}.
\end{align*}
Hence, $J_{\lambda} ( t \, u_{0}^{-} ) < \inf_{u \, \in \, \mathcal{N}_{\lambda}^{-}} J_{\lambda} ( u ) = \theta_{\lambda}^{-},$ which is a contradiction. Consequently $$ u_{n}^{-} \longrightarrow u_{0}^{-} \,\, \mbox{in} \,\, W^{1, q( x )}_{0} ( M ) \,\, \mbox{and} \,\, \lim_{n \rightarrow +\infty} J_{\lambda} ( u_{n}^{-} ) = J_{\lambda} ( u_{0}^{-} ) = \theta_{\lambda}^{-}. $$
Then, we conclude that $u_{0}^{-}$ is a minimizer of $J_{\lambda} $ on $ \mathcal{N}_{\lambda}^{-}.$
\end{proof}
\begin{theorem}\label{theorem4.3}
Under assumptions $( f_{1} ) - ( f_{3} )$ we assume that the smooth complete compact Riemannian n-manifold $( M, \, g )$ has property $B_{vol} ( \lambda, \, v )$. Then, there exists $ \bar{\lambda} $ such that for all $ \lambda \in ( 0, \, \bar{\lambda} ),$ the problem $( \mathcal{P} )$ has at least two non-negative weak solutions.
\end{theorem}
\begin{proof}
Form Theorems \ref{theorem 4.1} and \ref{theorem4.2}, we deduce that for any $ \lambda \in ( 0, \bar{\lambda} ), $ there exist $ u_{0}^{+} \in \mathcal{N}_{\lambda}^{+}$ and $ u_{0}^{-} \in \mathcal{N}_{\lambda}^{-}$ such as
$$ J_{\lambda} ( u_{0}^{+} ) = \inf_{u \, \in \, \mathcal{N}_{\lambda}^{+}} J_{\lambda} ( u ) \,\, \mbox{and} \,\, J_{\lambda} ( u_{0}^{-} ) = \inf_{u \, \in \, \mathcal{N}_{\lambda}^{-}} J_{\lambda} ( u ).$$
Then, the problem $ ( \mathcal{P} )$ has two solutions $ u_{0}^{+} \in \mathcal{N}_{\lambda}^{+}$ and $ u_{0}^{-} \in \mathcal{N}_{\lambda}^{-}$ in $ W_{0}^{1, q( x )} ( M ).$ By Lemma \ref{lem 3.3}, it follows that $ \mathcal{N}_{\lambda}^{-} \cap \mathcal{N}_{\lambda}^{+} = \emptyset.$ Then, $ u_{0}^{-} \neq u_{0}^{+}.$ Thus these two solutions are distinct.\\
Next, we prove that $u_{0}^{-}$ and $u_{0}^{+}$ are non-negative in $M$. For this, we introduce the truncation function $ h_{+} : M \times \mathbb{R} \longrightarrow \mathbb{R}$ defined by
$$ h_{+} ( x, s ) = \begin{cases}
0 & \text{ if $s < 0$ }, \\[0.3cm]
h( x, s ) & \text{if $ s \geq 0$ }.
\end{cases}$$
We set $ H_{+} ( x, s ) = \displaystyle\int_{0}^{s} f( x, t ) \,\, dt$ and consider the $ C^{1}-$functional\\ $ J_{\lambda}^{+} : W_{0}^{1, q( x )} ( M ) \longrightarrow \mathbb{R}$ given by
\begin{align*}
J_{\lambda}^{+} ( u ) =&  \int_{M} \frac{1}{p( x )} \, |\, \nabla u( x )\,|^{p( x )} \,\, dv_{g} ( x ) +  \int_{M} \frac{\mu ( x )}{q( x )} \, |\, \nabla u( x )\,|^{q( x )} \,\, dv_{g} ( x ) \\&- \int_{M} H_{+} ( x, u( x ) ) \,\, dv_{g} ( x ).
\end{align*}
Then, by \eqref{Said} and proposition \ref{prop13} we have for all $ u_{-} = \min \, \{\, 0, \, u( x ) \, \} $ that
\begin{align*}
0 = \langle ( J_{\lambda}^{+} )^{'} ( u_{-} ), \, u_{-} \rangle \, & \geq p^{-} \, \rho_{p( . )} ( |\, \nabla u_{-} \,| ) + \frac{\mu_{0}}{D^{p^{+}} ( c + 1 )^{p^{+}}} \, q^{-} \, \rho_{p( . )} ( u_{-} ) \\& \geq \rho_{p( . )} ( u_{-} ) \geq ||\, u_{-} \,||^{p^{-}}.
\end{align*}
Hence, $||\, u_{-} \,|| = 0,$ and thus $u = u_{+}$. Then, by taking $u = u_{0}^{-}$ and $u = u_{0}^{+}$ respectively, we deduce that $u_{0}^{-}$ and $ u_{0}^{+}$ are non-negative solutions of problem $(  \mathcal{P} )$.
\end{proof}
\textbf{Conclusion:} According to the above results, we can then say that $u^{\pm} $ are critical points of $J_{\lambda} $ and hence are non-negative weak solutions of problem $ ( \mathcal{P} ).$
 
\section*{Acknowledgements}
This paper has been supported by the RUDN University Strategic Academic Leadership Program and P.R.I.N. 2019.


%

%


%
%



\end{document}